\documentclass[11pt]{amsart}

\usepackage[margin=1.15in]{geometry}
\usepackage{amscd,amssymb, amsmath, setspace, wasysym}
\usepackage{graphicx}
\usepackage[all]{xy}

\usepackage{url}
\usepackage{hyperref}

\theoremstyle{plain}
\newtheorem{theorem}{Theorem}[section]
\newtheorem{prop}[theorem]{Proposition}

\newtheorem{lemma}[theorem]{Lemma}

\theoremstyle{definition}

\newtheorem*{ex*}{Example}

\newcommand\sO{{\mathcal O}}
\newcommand\sK{{\mathcal K}}
\newcommand\sH{{\mathcal H}}

\newcommand\sF{{\mathcal F}}
\newcommand\sG{{\mathcal G}}
\newcommand\sE{{\mathcal E}}
\newcommand\sI{{\mathcal I}}

\newcommand\sJ{{\mathcal J}}

\newcommand{\ddim}{{\rm dim}\,}
\newcommand{\pic}[1]{{\rm Pic}^0(#1)}

\newcommand\rp{{\mathbf{P}}}

\title[Regularity of curves in abelian varieties]
{Regularity of curves in abelian varieties}

\author{Luigi Lombardi and Wenbo Niu}
\address{Department of Mathematics, Statistics, and Computer Science\\ University of Illinois at Chicago, 851 S. Morgan Street, Chicago, IL, 60607}
 \email{\url{lombardi@math.uic.edu}}

\address{Department of Mathematics\\ Purdue University, 150, N. University Street, West Lafayette, IN, 47907}
\email{\url{niu6@math.purude.edu}}

\begin{document}
\begin{abstract}
Inspired by a theorem of Gruson-Lazarsfeld-Peskine 
bounding the Castelnuovo-Mumford regularity of curves in projective spaces, we bound the Theta-regularity
of curves in polarized abelian varieties.
\end{abstract}
\maketitle
\section{Introduction}
A coherent sheaf $\sF$ on a projective space $\rp^n$ is \emph{Castelnuovo-Mumford} $k$\emph{-regular} if for all $i>0$ the spaces 
$H^i(\rp^n,\sF(k-i))=0$. One can read off some geometric properties of a subvariety $Y\subset \rp^n$ just by looking at the regularity of 
its ideal sheaf $\sI_Y$. For instance, if $\sI_Y$ is Castelnuovo-Mumford $k$-regular, then 
$Y$ is cut-out by hypersurfaces of degree $k$ (\emph{cf}. \cite{La} Theorem 1.8.3). 
In this direction, a theorem of Gruson-Lazarsfeld-Peskine (\emph{cf}. \cite{GLP} Theorem 1.1), answering
and extending a classical question of Castelnuovo, turns out to be very useful: 
if $C$ is a reduced, irreducible, non-degenerate curve of degree $d$ in $\rp^n$, then $\sI_C$ is Castelnuovo-Mumford $(d+2-n)$-regular.

In analogy to Castelnuovo-Mumford regularity, Pareschi and Popa
introduced a notion of regularity for sheaves on a polarized abelian variety $(X,\Theta)$, the so called
$\Theta$-\emph{regularity} (\emph{cf}. \cite{PP1} Definition 6.1). It is defined as follows.
Given a coherent sheaf $\sF$ on $X$,
we denote by
\begin{equation}\label{V}
V^i(\sF):=\{\alpha\in{\rm Pic}^0(X)\,|\,h^i(X,\sF\otimes \alpha)>0\}
\end{equation}
the \emph{cohomological support loci of} $\sF$. Then we 
say that $\sF$ is \emph{Mukai-regular} (or \emph{M-regular} for short) if ${\rm codim}_{\pic{X}}\, V^i(\sF)>i$ for all $i>0$ and that
$$\sF\quad \mbox{ is } \quad k\mbox{-}\Theta\emph{-regular}\quad \stackrel{{\rm def}}{\Longleftrightarrow}
\quad \sF\otimes \Theta^{\otimes (k-1)}\quad \mbox{ is } \quad M\mbox{-regular}.$$
The systematic study of Pareschi and Popa on $\Theta$-regularity shows that $\Theta$-regular sheaves
enjoy analogous properties to Castelnuovo-Mumford regular sheaves on projective spaces (\emph{cf}. \cite{PP1} Theorem 6.3).
In particular, if $\sI_C$ is $k$-$\Theta$-regular, then $C$ is cut-out by $k$-$\Theta$-equations.

The aim of these notes is to provide a bound for the $\Theta$-regularity of curves in a polarized abelian variety.
\begin{theorem}\label{intr-thm}
Let $X$ be a complex abelian variety of dimension $n$ and $\Theta$ be an ample and globally generated line bundle on $X$. 
Let $\iota:C\hookrightarrow X$
be a reduced and irreducible curve and let $\nu:\widetilde{C}\longrightarrow C$ be its normalization. Set $f:=\iota\circ \nu$
and $d:={\rm deg}\, f^*\Theta$.
Then the ideal sheaf 
$$\sI_C\quad \mbox{ is }\quad \big(n+d+1\big)\mbox{-}\Theta\mbox{-regular}.$$
\end{theorem}
Since the square of an ample line bundle on an abelian variety is globally generated, we obtain the bound $(2n+4d+1)$ in
case the polarization $\Theta$ is not globally generated. 
As in the bound of Gruson-Lazarsfeld-Peskine, we note that our bound is linear and depends only on the dimension of the 
ambient space and on the degree of the curve.
On the other hand it is not sharp, as previous computations of $\Theta$-regularity
show that both Abel-Jacobi and Abel-Prym curves are $3$-$\Theta$-regular (\cite{PP1} Theorem 4.1, \cite{PP2} Theorem 7.17 and
\cite{CMLV} Corollary B). 
We point out that other bounds for $\Theta$-regularity have been worked out in \cite{PP1} Theorem 6.5
for subvarieties of a polarized abelian variety $(X,\Theta)$ defined by $d$-$\Theta$-equations.

The proof of Theorem \ref{intr-thm} goes as follows. As in \cite{GLP} Theorem 1.1, we use Eagon-Northcott complexes
to resolve $\sI_C$ with a complex of locally free sheaves such that: 1) it is exact away from $C$; 2) the Theta-regularity of its terms is 
easily computable. 
However, the methods to establish this resolution differ considerably from the ones used by Gruson-Lazarsfeld-Peskine as they mainly rely
on the generic vanishing theory developed in \cite{PP1} and \cite{PP2}.

\section{Setting and Proof}
Throughout the paper we work in the following setting. Let $X$ be a complex abelian variety of dimension $n$ and  $\Theta$ 
be an ample and globally generated line bundle on $X$. 
Let $\iota:C\hookrightarrow X$ be a reduced and irreducible curve of geometric genus $g$ and $\nu:\widetilde{C}\longrightarrow C$ be
its normalization. Set $f=\iota \circ \nu$ and $d:=\deg f^*\Theta$.
Let $\Gamma \subset \widetilde{C}\times X$ be the graph of $f$ and $p$ and $q$ be the projections from $\widetilde{C}\times X$ 
onto the first and second factor 
respectively. Finally we denote by $\sI_{C}$ and $\sI_{\Gamma}$ the ideal sheaves of $C$ and $\Gamma$ respectively.

We recall that a coherent sheaf $\sF$ on $X$ is \emph{continuously globally generated} if
there exists a positive integer $N$ such that for any general $\alpha_1,\ldots ,\alpha_N\in \pic{X}$ the sum of the twisted evaluation maps
$$\bigoplus_{i=1}^N H^0(X,\sF\otimes \alpha_i)\otimes \alpha_i^{-1}\longrightarrow \sF$$ is surjective. 
For instance, ample line bundles and, in general, $M$-regular sheaves are continuously globally generated (\emph{cf}. \cite{PP1} Proposition 2.13).

In order to compute the $\Theta$-regularity of $\sI_C$, we will show that 
it is enough to check the continuous global generation of sheaves of type
$q_*(\sI_{\Gamma}\otimes p^*A)\otimes \Theta$ where $A$ is a globally generated
line bundle on $\widetilde{C}$ (\emph{cf}. in Proposition \ref{en}). 
To begin with we present a simple lemma. 

\begin{lemma}\label{lem}
Let
$$\sE^{\bullet}: \cdots \longrightarrow \sE_2\stackrel{d_2}{\longrightarrow} \sE_1\stackrel{d_1}{\longrightarrow} \sE_0
\stackrel{d_0}{\longrightarrow} \sJ\longrightarrow 0$$
be a complex of coherent sheaves on a smooth projective irregular variety $Y$ such that $d_0$ is a surjective morphism.
If $\sE^{\bullet}$ is exact away from an algebraic set of dimension $\leq 1$, then we have an inclusion of 
cohomological support loci (\emph{cf}. \eqref{V})
$$V^i(\sJ)\subset V^i(\sE_0)\cup V^{i+1}(\sE_1)\cup\ldots\cup V^{\ddim Y}(\sE_{\ddim Y-i})\quad \mbox{ for any } \quad i\geq  1.$$
\end{lemma}
\begin{proof}
We set $\sK_i:={\rm ker}\,d_i$, $\sI_i:={\rm im}\, d_i$ and 
$\sH_i:=\sK_i/\sI_{i+1}$ for $i\geq 0$.
By assumption $\ddim {\rm Supp}\, \sH_i\leq 1$ for any $i\geq 0$. Therefore
$$V^j(\sH_i)=\emptyset \quad \mbox{ for any }\quad i\geq 0\quad \mbox{ and } \quad j>1.$$
By looking at the exact sequence $0\longrightarrow \sK_0\longrightarrow \sE_0\longrightarrow \sJ\longrightarrow 0$,
we have $$V^i(\sJ)\subset V^i(\sE_0)\cup V^{i+1}(\sK_0).$$ In addition,
the exact sequence $0\longrightarrow \sI_1\longrightarrow \sK_0\longrightarrow \sH_0\longrightarrow 0$ yields
$$V^{i+1}(\sK_0)\subset V^{i+1}(\sI_1)\cup V^{i+1}(\sH_0)=V^{i+1}(\sI_1).$$
Finally, by looking at the exact sequence $0\longrightarrow \sK_1\longrightarrow \sE_1\longrightarrow \sI_1\longrightarrow 0$,
we deduce that $$V^{i+1}(\sI_1)\subset V^{i+1}(\sE_1)\cup V^{i+2}(\sK_1)$$ and therefore that
$$V^i(\sJ)\subset V^i(\sE_0)\cup V^{i+1}(\sE_1)\cup V^{i+2}(\sK_1).$$
At this point it is enough to iterate the previous argument to obtain the lemma.
\end{proof}

\begin{prop}\label{en}
Let $A$ be a globally generated line bundle on $\widetilde{C}$.
If $q_*(\sI_{\Gamma}\otimes p^*A)\otimes \Theta$ is continuously globally generated, then 
$$\sI_C\quad \mbox{ is }\quad \big(h^0(\widetilde{C},A)+n\big)\mbox{-}\Theta\mbox{-}regular.$$
\end{prop}
\begin{proof}
 Consider the exact sequence defining the graph $\Gamma$
\begin{equation}\label{gamma}
0\longrightarrow \sI_{\Gamma}\longrightarrow \sO_{\widetilde{C}\times X}\longrightarrow \sO_{\Gamma}\longrightarrow 0.
\end{equation}
By tensoring \eqref{gamma} by $p^*A$ and then by pushing forward to $X$ via $q$, we obtain the exact sequence
\begin{equation}\label{ex}
0\longrightarrow q_*(p^*A\otimes \sI_{\Gamma})\longrightarrow H^0(\widetilde{C},A)\otimes \sO_X\stackrel{{\rm ev}}{\longrightarrow} f_*A.
\end{equation}
Let $\sG$ be the image of the evaluation map ${\rm ev}$.
By hypotheses, there exists
a positive integer $N$ and line bundles $\alpha_1,\ldots ,\alpha_N\in {\rm Pic}^0(X)$ such that the map
\begin{equation}\label{cggq}
\bigoplus_{i=1}^N H^0(X,q_*(p^*A\otimes \sI_{\Gamma})\otimes \Theta\otimes \alpha_i)\otimes \alpha_i^{-1}
\longrightarrow q_*(p^*A\otimes \sI_{\Gamma})\otimes \Theta 
\end{equation}
is surjective. We set 
$W_i:=H^0(X,q_*(p^*A\otimes \sI_{\Gamma})\otimes \Theta\otimes \alpha_i)$ for $i=1,\ldots, N$.  From \eqref{ex} and \eqref{cggq},
we deduce 
a presentation of $\sG$: 
\begin{equation}\label{succ}
\bigoplus_{i=1}^N W_i\otimes \alpha_i^{-1}
\otimes \Theta^{-1}\stackrel{\varphi}{\longrightarrow} H^0(\widetilde{C},A)\otimes \sO_X\longrightarrow \sG\longrightarrow 0.
\end{equation}
We set $$E:=\bigoplus_{i=1}^N W_i\otimes \alpha_i^{-1}\otimes \Theta ^{-1}\quad \mbox{ and } \quad h:=\ddim H^0(\widetilde{C},A).$$
Let $\sJ$ be the $0$-th Fitting ideal of $\sG$. Since $A$ is globally generated and $\sI_C$ is a radical ideal, 
$\sJ$ coincides with $\sI_C$ away the singular points of $C$ (see \cite{GLP} p.496).
By applying the Eagon-Northcott complex (\emph{cf}. \cite{GLP} (0.4)) to the morphism $\varphi$ in \eqref{succ}, we get a complex
$$E^{\bullet}:\cdots\longrightarrow E_1\longrightarrow E_0\longrightarrow \sJ\longrightarrow 0$$
which is exact away from $C$ and whose terms are copies of
$\bigwedge ^{h+i}E$ for all $i\geq 0$. Therefore 
$$E_i\cong \bigoplus_{t} \Theta^{\otimes (-h-i)}\otimes \beta_{i_t}\quad  \mbox{ for some }\quad \beta_{i_t}\in \pic{X}.$$
Moreover we have
$$V^j(E_i\otimes \Theta^{\otimes (h+n-1)})\subset V^j(\Theta^{\otimes (n-i+1)})=\emptyset\quad \mbox{ for any }\quad j>0\quad \mbox{ and }
\quad i<n-1$$
and $$V^j(E_{n-1}\otimes \Theta^{\otimes (h+n-1)})\subset V^j(\sO_X)=\{\sO_X\}\quad \mbox{ for any }\quad j>0.$$
Thus, by Lemma \ref{lem} we obtain inclusions
$$V^1(\sJ\otimes \Theta^{\otimes (h+n-1)})\subset V^n(E_{n-1}\otimes \Theta^{\otimes (h+n-1)})\subset \{\sO_X\}$$ and
$$V^j(\sJ\otimes \Theta^{\otimes (h+n-1)})=\emptyset\quad \mbox{ for any }\quad j>1.$$
Finally, by noting that 
$$V^j(\sI_C\otimes \Theta^{\otimes (h+n-1)})\subset V^j(\sJ\otimes \Theta^{\otimes (h+n-1)})\quad \mbox{ for any }\quad j>0,$$ we conclude then 
that $\sI_C\otimes \Theta^{\otimes (h+n-1)}$ is an $M$-regular sheaf.
\end{proof}

In the next proposition we
will give sufficient conditions for the hypotheses of Proposition \ref{en} to be satisfied.
\begin{prop}\label{pencil}
If $B$ is a line bundle on $\widetilde{C}$ such that $f_*B$ is $M$-regular on $X$, then
$q_*(\sI_{\Gamma}\otimes p^*(B\otimes f^*\Theta))\otimes \Theta$ is $M$-regular on $X$.
\end{prop}

\begin{proof} 
The sheaf $f_*B\otimes \Theta^{\otimes m}$ is globally generated for any $m\geq 1$ by \cite{PP1} Proposition 2.12. 
Moreover, by \cite{PP2} Proposition 3.1, $f_*B\otimes \Theta^{\otimes m}$ is
an $I.T.0$ sheaf for any $m\geq 1$, i.e.  
$V^i(f_*B\otimes \Theta^{\otimes m})=\emptyset$ for all $i,m\geq 1$.

By \eqref{ex} we obtain exact sequences for any $\alpha\in \pic{X}$
\begin{equation}\label{exgamma3}
 0\longrightarrow q_*(\sI_{\Gamma}\otimes p^*(B\otimes f^*\Theta))\otimes \Theta\otimes \alpha\longrightarrow
H^0(\widetilde{C},B\otimes f^*\Theta)\otimes
\Theta \otimes \alpha
\longrightarrow f_*B\otimes \Theta^{\otimes 2}\otimes \alpha\longrightarrow 0.
\end{equation}
We set $\sH:=q_*(\sI_{\Gamma}\otimes p^*(B\otimes f^*\Theta))$ so that we only need to check the conditions
${\rm codim}_{\pic{X}}V^i(\sH\otimes \Theta)>i$ for all $i>0$.
By the Kodaira Vanishing Theorem and by \eqref{exgamma3}, we have
$$V^i(\sH\otimes \Theta)=\emptyset \quad \mbox{ for any }\quad i\geq 3.$$ Furthermore
$$V^2(\sH\otimes \Theta)\cong V^1(f_*B\otimes \Theta^{\otimes 2})=\emptyset.$$
Now we study the dimension of $V^1(\sH\otimes \Theta)$.
By denoting by
\begin{equation}\label{mult}
m_{\alpha}: H^0(X,f_*B\otimes \Theta)\otimes H^0(X,\Theta\otimes \alpha)\longrightarrow
H^0(X,f_*B\otimes \Theta^{\otimes 2}\otimes \alpha)
\end{equation}
the multiplication map on global sections induced by \eqref{exgamma3}, we have an identification
\begin{eqnarray}\label{V1}
V^1(\sH\otimes \Theta)=\{\alpha\in{\rm Pic}^0(X)\,|\,m_{\alpha}\mbox{ is not surjective}\}.
\end{eqnarray}
We claim that the inverse morphism $(-1):{\rm Pic}^0(X)\longrightarrow {\rm Pic}^0(X)$ taking $\alpha$ to $\alpha^{-1}$
maps
\begin{equation}\label{inclusion}
 V^1(\sH\otimes \Theta)\mapsto V^1(f_*B).
\end{equation}
This finishes the proof since it implies
$$\ddim V^1(\sH\otimes \Theta)\leq\ddim V^1(f_*B)\leq n-2.$$ 

Now we show \eqref{inclusion}.
Since $\Theta\otimes \alpha$ is globally generated for any $\alpha\in \pic{X}$, we have exact sequences
$$0\longrightarrow M_{\Theta\otimes \alpha}\longrightarrow H^0(X,\Theta\otimes \alpha)\otimes \sO_X\longrightarrow
\Theta\otimes \alpha\longrightarrow 0.$$
Tensoring by $\Theta$ and then restricting to $C$ and finally tensoring by $\nu_*B$, we get exact sequences
$$0\longrightarrow \iota^*(M_{\Theta\otimes \alpha}\otimes \Theta)\otimes \nu_*B
\longrightarrow H^0(X,\Theta\otimes \alpha)\otimes \iota^*\Theta\otimes \nu_*B
\stackrel{{\rm ev}_{\alpha}}{\longrightarrow} \iota^*(\Theta^{\otimes 2}\otimes \alpha)\otimes \nu_*B\longrightarrow 0.$$
We note that the map on global sections induced by ${\rm ev}_{\alpha}$ coincides with the multiplication map \eqref{mult}. Hence
by \eqref{V1}
\begin{equation}\label{incl2}
\alpha\in V^1(\sH\otimes \Theta)\quad \Longrightarrow \quad
H^1(C,\iota^*(M_{\Theta\otimes \alpha}\otimes \Theta)\otimes \nu_*B)\neq 0.
\end{equation}
Now pick an arbitrary element $\alpha \in V^1(\sH\otimes \Theta)$ and
set $W:={\rm Im}\big(H^0(X,\Theta\otimes \alpha)\longrightarrow H^0(C,\iota^*(\Theta\otimes \alpha))\big)$. We note that
$W$ generates $\iota^*(\Theta\otimes \alpha)$ and hence $\ddim W\geq 2$ since $\Theta$ is not trivial.
Moreover, the preimages $s_1$ and $s_2$ in
$H^0(X,\Theta\otimes \alpha)$ of two general sections in $W$ generate $\iota^*(\Theta\otimes \alpha)$.
We have then a commutative diagram

\centerline{ \xymatrix@=32pt{
 & 0  \ar[r] & \iota^*(\Theta^{-1}\otimes \alpha^{-1})\ar[d]\ar[r] & \sO_{C}\oplus
\sO_{C}\ar[d]\ar[r]^{(s_1,s_2)} & \iota^*(\Theta
\otimes \alpha) \ar @{=}[d]\ar[r] & 0\\
& 0 \ar[r] & \iota^*M_{\Theta\otimes \alpha}  \ar[r]  & H^0(X,\Theta\otimes \alpha)\otimes \sO_{C}
\ar[r] & \iota^*(\Theta\otimes
\alpha) \ar[r] & 0.\\ }}
\noindent
By defining $\widetilde{V}:=H^0(X,\Theta\otimes \alpha)/ \langle s_1,s_2\rangle$
and by the Snake Lemma, we obtain the exact sequence
$$0\longrightarrow \iota^*(\Theta^{-1}\otimes \alpha^{-1})\longrightarrow \iota^*M_{\Theta\otimes \alpha}
\longrightarrow \widetilde{V}\otimes \sO_C\longrightarrow 0$$ and hence the exact sequence
$$0\longrightarrow  \iota^*\alpha^{-1}\otimes \nu_*B
\longrightarrow \iota^*(M_{\Theta\otimes \alpha}\otimes \Theta)\otimes \nu_*B
\longrightarrow \widetilde{V}\otimes \iota^*\Theta\otimes \nu_*B\longrightarrow 0.$$
Finally, the projection formula yields isomorphisms 
$$H^1(C,\nu_*B\otimes \iota^*\alpha^{-1})\cong H^1(X,f_*B\otimes \alpha^{-1})\quad \mbox{ and }
\quad H^1(C,\nu_*B\otimes \iota^*\Theta)\cong H^1(X,f_*B\otimes \Theta)=0$$ which in turn show that 
$\alpha^{-1}\in V^1(f_*B)$ by \eqref{incl2}. 
\end{proof}
Before proceeding with the proof of Theorem \ref{intr-thm}, we prove a lemma giving a necessary condition for a sheaf of the form $f_*B$ to be 
$M$-regular on $X$.

\begin{lemma}\label{intersection}
 Let $D$ be a smooth and irreducible curve of genus $g$ and $\varphi:D\longrightarrow X$ be a morphism to
a complex abelian variety $X$ of dimension $n$.
If $B$ is a general line bundle on $D$ of degree $b$,
then 
$$\ddim V^1(\varphi_*B)\leq n+g-b-2.$$
\end{lemma}
\begin{proof}
Without loss of generality, we can assume that $\varphi$ in non-constant since in this case 
$V^1(\varphi_*B)=\emptyset$ for all $B\in \pic{D}$.
The algebraic set $V^1(B)$ is irreducible as Serre duality yields an isomorphism $V^1(B)\cong W_{2g-2-b}(D)$ 
(here $W_{2g-2-b}(D)$ is the image of the Abel-Jacobi map
${\rm Sym}^{2g-2-b}(D)\longrightarrow {\rm Pic}^{2g-2-b}(D)$). 
The algebraic group $\pic{D}$ acts on itself via translations. For any $\gamma \in \pic{D}$, we write $\gamma V^1(B)$ for 
the image of $V^1(B)$ under the action of $\gamma$. 
Then, by Kleiman's Transversality Theorem \cite{Kl} Theorem 2, there exists an open dense subset $V\subset \pic{D}$ such that for all $\gamma \in V$
the fiber product 
$\gamma V^1(B)\times_{\pic{D}}\pic{X}$ is either empty or of dimension $$\ddim V^1(B)+\dim \pic{X}-\dim \pic{D}=\ddim V^1(B)+n-g.$$
We note the isomorphisms of algebraic sets $\gamma V^1(B)\cong V^1(B\otimes \gamma^{-1})$ for any $\gamma \in V$.
Moreover, by the universal property of the fiber product and by the projection formula, we obtain closed 
immersions $V^1(\varphi_*(B\otimes\gamma^{-1}))\hookrightarrow V^1(B\otimes \gamma^{-1})\times_{\pic{D}}\pic{X}$.
Hence for any $\gamma\in V$ we have
\begin{eqnarray*}
\dim V^1(\varphi_*(B\otimes \gamma^{-1})) & \leq & \ddim \big(\gamma V^1(B)\times_{\pic{D}}\pic{X}\big)\\
& = & \ddim V^1(B)+n-g\\
& = & \ddim W_{2g-2-b}(D)+n-g\\
& \leq & n+g-b-2.
\end{eqnarray*}
\end{proof}

At this point the proof of Theorem \ref{intr-thm} easily follows by the previous lemmas and propositions.
\begin{proof}[Proof of Theorem \ref{intr-thm}.]
By Lemma \ref{intersection} we can pick a line bundle $B$ of degree $g$ on $\widetilde C$ such that $f_*B$ is $M$-regular on $X$ and 
$h^1(\widetilde{C},B\otimes f^*\Theta)=0$.
Hence, the sheaf $q_*(\sI_{\Gamma}\otimes p^*(B\otimes f^*\Theta))\otimes \Theta$ is
continuously globally generated by Proposition \ref{pencil} and
we conclude then by applying Proposition \ref{en} after having noted that $h^0(\widetilde{C},B\otimes f^*\Theta)=
d+1$. 
\end{proof}

\noindent \textbf{Acknowledgements.}
We thank Lawrence Ein, Angela Ortega, Giuseppe Pareschi and Mihnea Popa for stimulating and useful conversations. 
We also thank the Mathematics Research Communities (a program of the AMS) for having supported a visit of the second author by the first.

\end{document}